\documentclass[11pt,twoside]{article}%{cctart}

\usepackage{amsfonts}
\usepackage{amsmath}
\usepackage{amsthm}
\usepackage[mathscr]{euscript}
\usepackage{mathrsfs}
\usepackage{indentfirst}
\usepackage{xcolor}

\allowdisplaybreaks[4] %%让多行公式强制换页%%

\newtheorem{theo}{Theorem}[section]

\newtheorem{lem}{Lemma}[section]
\newtheorem{remark}{Remark}[section]

\newcommand{\be}{\begin{equation}}
\newcommand{\ee}{\end{equation}}
\newcommand\bes{\begin{eqnarray}} \newcommand\ees{\end{eqnarray}}
\newcommand{\bess}{\begin{eqnarray*}}
\newcommand{\eess}{\end{eqnarray*}}

\newcommand\ep{\varepsilon}
\newcommand\qq{\eqref}
\newcommand\kk{\left}
\newcommand\rr{\right}
\newcommand\dd{\displaystyle}
\newcommand\dx{{\rm d}x}

\newcommand\ds{{\rm d}s}

\newcommand\dv{\frac{d}{dt}}
\newcommand\vp{\varphi}
\newcommand\df{\dd\frac}

\newcommand\kp{\kappa}
\newcommand\lm{\lambda}

\newcommand\ii{\int_\oo}

\newcommand\yy{\infty}

\newcommand\R{\mathbb{R}}

\newcommand\oo{\Omega}
\newcommand\tr{\Delta}

\newcommand\nn{\nabla}
\newcommand\nm{\nonumber}
\newcommand\ma{\Sigma}

\newcommand\pl{\partial}
\newcommand\tk{{\tilde k}}

\newcommand\tf{{\tilde T}}

\newcommand\te{\theta}

\newcommand\mb{\Gamma}

\begin{document}
\begin{center}{\Large\bf Global bounded solution of the higher-dimensional}\\[2mm]
{\Large\bf forager-exploiter model with/without growth sources}\footnote{This work was
supported by NSFC Grants 11771110, 11971128}\\[4mm]
 {\Large  Jianping Wang, \ \ Mingxin Wang\footnote{Corresponding author. {\sl E-mail}: mxwang@hit.edu.cn}}\\[1mm]
{School of Mathematics, Harbin Institute of Technology, Harbin 150001, PR China}
\end{center}

\begin{quote}
\noindent{\bf Abstract.} We investigate the higher-dimensional forager-exploiter model with homogeneous Neumann boundary condition
\bess
 \left\{\begin{array}{lll}
 u_t=\tr u-\chi\nn\cdot(u\nn w)+\eta_1(u-u^m),&x\in\oo,\ \ t>0,\\[1mm]
 v_t=\tr v-\xi\nn\cdot(v\nn u)+\eta_2(v-v^l),&x\in\oo,\ \ t>0,\\[1mm]
 w_t=\tr w-\lm(u+v)w-\mu w+r(x,t),&x\in\oo,\ \ t>0,\\[1mm]
 u(x,0)=u_0(x),\ v(x,0)=v_0(x),\ w(x,0)=w_0(x), &x\in\oo,
 \end{array}\right.
 \eess
where $\oo\subset\R^n$ is a bounded domain, the constants $\chi,\xi,\lm,\mu$ are positive, $m, l>1$ and $\eta_1,\eta_2\ge0$, $r\in C^1(\bar\oo\times[0,\yy))\cap L^\yy(\oo\times(0,\yy))$. The nonnegative initial functions $u_0,w_0\in W^{2,\yy}(\oo)\setminus \{0\}$ and $v_0\in W^{1,\yy}(\oo)\setminus \{0\}$. It will be shown that, for $\eta_1=\eta_2=0$ and $n\ge2$, on the one hand, there is $\ep>0$ such that, if
\[\|u_0\|_{W^1_{2p}(\oo)}+\|w_0\|_{W^1_{2(p+1)}(\oo)}+\|r\|_{L^\yy(\oo\times(0,\yy))}\le \ep\ \ {\rm with}\ \ p=\min\kk\{k\in\mathbb{N}:\ k>{n}/{2}\rr\},\]
then this system is globally solvable in the classical sense; On the other hand, the weak taxis effects rule out any blow-up. Secondly, for $\eta_1>0, m\ge2,\eta_2=0$ and $n=2$, we establish the global solvability of this problem provided only $\xi$ is sufficiently small. Finally, when $n=2$, $\eta_1,\eta_2>0$, we find a condition for the logistic degradation rates $m, l$ that ensures the global existence of the classical solutions.


 \end{quote}

 \section{Introduction}
 \setcounter{equation}{0} {\setlength\arraycolsep{2pt}
This paper concerns the higher-dimensional forager-exploiter model with different taxis strategies for two groups in search of food (\cite{Blcak-Arxiv2019,Tania2012,TaoW-2019-forager}):
\bes
 \left\{\begin{array}{lll}
 u_t=\tr u-\chi\nn\cdot(u\nn w)+\eta_1(u-u^m),&x\in\oo,\ \ t>0,\\[1mm]
 v_t=\tr v-\xi\nn\cdot(v\nn u)+\eta_2(v-v^l),&x\in\oo,\ \ t>0,\\[1mm]
 w_t=\tr w-\lm(u+v)w-\mu w+r(x,t),&x\in\oo,\ \ t>0,\\[1mm]
 \pl_\nu u=\pl_\nu v=\pl_\nu w=0,\ \ &x\in\pl\oo, \ \ t>0,\\[1mm]
 u(x,0)=u_0(x),\ v(x,0)=v_0(x),\ w(x,0)=w_0(x), &x\in\oo,
 \end{array}\right.\label{1.1}
 \ees
where $\oo\subset\R^n$ is a bounded domain with smooth boundary $\pl\oo$, $\pl_\nu$ denotes differentiation with respect to the outward normal vector $\nu$ on $\pl\oo$. The unknown functions $u,v$ and $w$ are the densities of the forager population, exploiter population and nutrient, respectively. The constants $\chi,\xi,\lm,\mu$ are supposed to be positive, $m,l>1$ and $\eta_1,\eta_2$ are nonnegative. The given nonnegative function $r(x,t)$ is the production rate of nutrient and satisfies
  \bes
  r\in C^1(\bar\oo\times[0,\yy))\cap L^\yy(\oo\times(0,\yy)).
  \label{x1}\ees
It is assumed that, besides the random diffusions, the first species (foragers) move toward the increasing nutrient gradient direction, while the second species (exploiters) follow the foragers to find the food indirectly.

The model \eqref{1.1} was initially proposed by Tania et al. \cite{Tania2012}. When $n=1$, $\eta_1=\eta_2=0$ and $r$ is a nonnegative constant, Tao and Winkler in \cite{TaoW-2019-forager} proved the global existence of the classical solution to \eqref{1.1}, and showed that the solution stabilizes to a positive constant equilibrium exponentially provided that $\min\{\ii u_0,\ii v_0\}$ is small enough. When $n\ge1$ and $\eta_1=\eta_2=0$, by taking into account the volume-filling effect (i.e., $u_t=\tr u-\nn\cdot(u(1-u)\nn w)$, $v_t=\tr v-\nn\cdot(v(1-v)\nn u)$ and $0\le u_0,v_0\le1$), the global existence of the classical solution was builded in \cite{Liuyuanyuan}. For $\eta_1=\eta_2=0$ and $n\ge1$, under an explicit condition linking $r$ and $w_0$, the global existence of generalized solutions of system \eqref{1.1} has been obtained and the long time behavior of the solution was established when $r$ decays suitably in time (\cite{Winkler-M3AS2019}). When $n=2$ and $\eta_1,\eta_2>0$, Black \cite{Blcak-Arxiv2019} showed that \eqref{1.1} has at least one global generalized solution if $m>\sqrt{2}+1$ and $\min\{m,l\}>(m+1)/(m-1)$, and proved that the generalized solution actually becomes a classical one after some waiting time $T_*$ (possibly large) under conditions: $m,l>\sqrt{2}+1$, $r\ge 0$ satisfies \eqref{x1} and $r\in L^1((0,\yy);\,L^\yy(\oo))$.

It is noticed that, none of the above mentioned works involve the global existence of classical solutions to \eqref{1.1} in high dimensions. The motivation of this paper is to establish the global solvability of \eqref{1.1} in high dimensions in the classical sense.

In order to better describe our contribution, we recall some related works on the original minimal Keller-Segel system in bounded domain
\bes
 \left\{\begin{array}{lll}
 u_t=\tr u-\nn\cdot(u\nn v),&x\in\oo,\ \ t>0,\\[1mm]
 v_t=\tr v-v+u,&x\in\oo,\ \ t>0,\\[1mm]
 \pl_\nu u=\pl_\nu v=0,\ \ &x\in\pl\oo, \ \ t>0,\\[1mm]
 u(x,0)=u_0(x),\ v(x,0)=v_0(x), &x\in\oo,
 \end{array}\right.\label{1.2}
 \ees
where $\oo\subset\R^n$ ($n\ge2$), $u_0\in C(\bar\oo)$ and $v_0\in W^1_\sigma(\oo)$ with $\sigma>n$. It was shown in \cite{Cao-DCDSA} that, there exist $\ep,\lm>0$ such that, if
\[\|u_0\|_{L^{n/2}(\oo)}\le\ep,\ \ \ \|v_0\|_{W^1_n(\oo)}\le\ep,\]
then \eqref{1.2} admits a unique nonnegative classical solution $(u,v)$ which exists globally in time and satisfies
\[\|u(\cdot,t)-\bar u_0\|_{L^\yy(\oo)}+\|v(\cdot,t)-\bar v_0\|_{L^\yy(\oo)}\le Ce^{-\lm t},\ \ t>0\]
with $\bar u_0=\frac1{|\oo|}\ii u_0$, $\bar v_0=\frac1{|\oo|}\ii v_0$ and some $C>0$. For the minimal Keller-Segel model in $\R^n$, the boundedness result holds provided $\|u_0\|_{L^r(\R^n)}$ and $\|v_0\|_{W^1_q(\R^n)}$ small enough with $r>n/2$ and $q\ge n$ (\cite{Corrias-Petthame}). For the prey-taxis system
\bes
 \left\{\begin{array}{lll}
 u_t=\tr u-\nn\cdot(u\nn v)+uv,&x\in\oo,\ \ t>0,\\[1mm]
 v_t=\tr v+v\kk(k-v\rr)-uv,&x\in\oo,\ \ t>0,\\[1mm]
 \pl_\nu u=\pl_\nu v=0,\ \ &x\in\pl\oo, \ \ t>0,\\[1mm]
 u(x,0)=u_0(x),\ v(x,0)=v_0(x), &x\in\oo,
 \end{array}\right.\label{1.3}
 \ees
it was shown in \cite{WjpW-ZAMP,W-S-W2017} that, under a small condition on $\max\{\|v_0\|_{L^\yy(\oo)},k\}$, the classical solution to \eqref{1.3} is global in time and bounded. Inspired by above mentioned works (the classical Keller-Segel model \cite{Cao-DCDSA,Corrias-Petthame} and prey-taxis system \cite{JinW, WjpW-ZAMP,W-S-W2017}), it is reasonable to conjecture the global existence of the classical solutions to \eqref{1.1} under some smallness conditions. However, unlike the simple structure in the minimal Keller-Segel system or the prey-taxis model, the different taxis strategies in the system \eqref{1.1} produce more mathematical challenges and make the analysis work more complicate and difficult.

In this paper we investigate the global bounded solution of \eqref{1.1} for the higher-dimensional case: $n\ge 2$. Throughout this paper, we use $C$, $C'$ and $C_i$ to represent generic positive constants, which may be different in different places. And, for simplicity, we use $\ii f$ and  $\|f\|_p$ to denote $\ii f\dx$ and $\|f\|_{L^p(\oo)}$, respectively. We suppose from now on that
   \[u_0,w_0\in W^2_{\yy}(\oo),\ v_0\in W^1_{\yy}(\oo),\ \ {\rm and}\ \ u_0,v_0,w_0\ge 0,\,\not\equiv 0\ \ {\rm in}\ \bar\oo.\]

Our first statement is that, for the fixed $\chi,\xi>0$, the small initial data can prevent blow-up.
\begin{theo}\label{t1.1}
Let $n\ge 2$, $\eta_1=\eta_2=0$ and $p=\min\kk\{k\in\mathbb{N}:\ k>{n}/{2}\rr\}$. For the fixed $\chi,\xi>0$, there exists $\ep>0$ such that, if
 \bess
 \|u_0\|_{W^1_{2p}(\oo)}+\|w_0\|_{W^1_{2(p+1)}(\oo)}+\|r(x,t)\|_{L^\yy(\oo\times(0,\yy))}\le\ep,
 \eess
then \qq{1.1} admits a unique nonnegative global bounded classical solution $(u,v,w)\in [C(\bar\Omega\times [0,\yy))\cap C^{2,1}(\bar\Omega\times (0,\yy))]^3$. Moreover, there exists $C>0$ such that
\bes
\|u(\cdot,t)\|_{W^1_{2p}(\oo)}+\|v(\cdot,t)\|_{W^1_{2p}(\oo)}+\|w(\cdot,t)\|_{W^1_{2p}(\oo)}\le C,\ \ \forall\ t\in(0,\yy).\label{1.6}
\ees
\end{theo}

The second conclusion is that, for the given initial data $(u_0,v_0,w_0)$, if both the taxis effects $\chi,\xi$ are weak enough, then the classical solution of \eqref{1.1} exists globally and remain bounded.
\begin{theo}\label{t1.1a}\, Let $n\ge2$ and $\eta_1=\eta_2=0$. For the given initial data $(u_0,v_0,w_0)$, there exists $\chi_0,\xi_0>0$ such that, if
\bess
\chi\le\chi_0,\ \ \ \xi\le\xi_0,
\eess
then \qq{1.1} admits a unique nonnegative global bounded classical solution $(u,v,w)\in [C(\bar\Omega\times [0,\yy))\cap C^{2,1}(\bar\Omega\times (0,\yy))]^3$ which satisfies \eqref{1.6}.
\end{theo}

It is noted that Theorems \ref{t1.1} and \ref{t1.1a} will be proved by the same arguments and both of them can be easily deduced from Lemma \ref{l3.6} below.

Next, we observe that, when $\eta_1>0,m\ge2$ and $\eta_2=0$, the solution exists globally and remain bounded in 2 dimension provided only $\xi$ is small (without any restriction on $\chi$).

\begin{theo}\label{t1.2}
Let $n=2$, $\eta_1>0,m\ge2$ and $\eta_2=0$. Then there exists $\xi_0>0$ such that, for $\xi\le\xi_0$, the problem \qq{1.1} admits a unique nonnegative global bounded classical solution $(u,v,w)\in [C(\bar\Omega\times [0,\yy))\cap C^{2,1}(\bar\Omega\times (0,\yy))]^3$ which satisfies, for some $C>0$,
\bes
\|u(\cdot,t)\|_{W^1_4(\oo)}+\|v(\cdot,t)\|_{W^1_4(\oo)}+\|w(\cdot,t)\|_{W^1_4(\oo)}\le C,\ \ \forall\ t\in(0,\yy).\label{1.7}
\ees
\end{theo}

The last aim of this article is to show that the suitable large damping rates can prevent blow-up in two dimensional case.

\begin{theo}\label{t1.3}
Let $n=2$ and $\eta_1,\eta_2>0$. If
\bes
m,l\ge2\ \ {\rm and}\ \ l\ge\max\kk\{3,3\tilde m/(2\tilde m-3)\rr\}\ \ {\rm with}\ \tilde m=\min\{m,l\},\label{1.8}
\ees
then the problem \qq{1.1} admits a unique nonnegative global bounded classical solution $(u,v,w)\in [C(\bar\Omega\times [0,\yy))\cap C^{2,1}(\bar\Omega\times (0,\yy))]^3$. Moreover, \eqref{1.7} holds as well.
\end{theo}

\begin{remark}\label{r1.1}
The condition \eqref{1.8} is equivalent to
\bes
2\le m<3,\ l\ge 3m/(2m-3)\ \ {\rm or}\ \ m,l\ge3.\label{1.9}
\ees
We recall from the pioneer work {\rm\cite{OsakiTY}} and recent researches {\rm\cite{Xiang-JMP,Zheng-Xiang Nonlinearity}} that the logistic growth or some sub-logistic source is enough to rule out any blow-up in the Keller-Segel systems in two dimensional case. It would be rather meaningful to investigate wether or not the condition \eqref{1.9} or \eqref{1.8} is the optimal one for ensuring the classical global solvability of \eqref{1.1} in two dimension. We leave this challenging problem as the future work.
\end{remark}

\begin{remark}\label{r1.2}
In the present paper, we use the special form of the generalized logistic source $\eta_1(u-u^m)$ and $\eta_2(v-v^l)$ for the simplicity. It seems worthwhile to mention that, for the generic constant parameters, i.e., $\eta_1(a_1u-b_1u^m)$ and $\eta_2(a_2v-b_2v^l)$ with $a_i,b_i>0$ for $i=1,2$, the conclusions in Theorem \ref{t1.2} and Theorem \ref{t1.3} hold as well.
\end{remark}

The article is organized as follows. Section 2 provides some basic preliminaries, including the local solvability of \eqref{1.1} and a fundamental ODE result . Section 3 is devoted to prove that the small values of initial data or taxis coefficients can prevent blow-up in any dimensions when $\eta_1=\eta_2=0$ (Theorem \ref{t1.1} and Theorem \ref{t1.1a}). In section 4, we show that, in the case $\eta_1>0, m\ge2$ and $\eta_2=0$, the small value of $\xi$ is enough to ensure the global existence of the solution to \eqref{1.1} (Theorem \ref{t1.2}). We finally prove in Section 5 that, when $\eta_1,\eta_2>0$ and logistic degradation rates $m,l$ are sufficiently large, then \eqref{1.1} is globally solvable (Theorem \ref{t1.3}).

\section{Local existence and preliminaries}
  \setcounter{equation}{0} {\setlength\arraycolsep{2pt}

The local solvability of \eqref{1.1} are well established by the Amann theory (cf. \cite[Lemma 2.1]{TaoW-2019-forager} or \cite{WjpW-JMP}). The positivity of solution and the uniform-in-time boundedness of $w$ are due to the parabolic maximum principle.

\begin{lem}\label{l2.1}\, Suppose that $n\ge1$, $\eta_i\ge0$ with $i=1,2$ and $m,l>1$. Then there exist a $T_m\in(0,\yy]$ and a unique solution $(u,v,w)$ which solves \qq{1.1} in $[0,T_m)$, $u,v,w\in C(\bar\Omega\times [0,T_m))\cap C^{2,1}(\bar\Omega\times (0,T_m))$ and satisfies
 \bes
 u,\,v>0, \ \ \ 0< w\le Q\ \ \ {\rm in} \ \ \bar\Omega\times (0,T_m).\label{2.1}
 \ees
Moreover, the ``existence time $T_m$" can be chosen maximal: either
$T_m=\infty$, or $T_m<\infty$ and
  \[\limsup_{t\to T_m}\big(\|u(\cdot,t)\|_{W^1_q(\oo)}+\|v(\cdot,t)\|_{W^1_q(\oo)}
  +\|w(\cdot,t)\|_{W^1_q(\oo)}\big)=\yy\ \ \ {\rm for\ all}\ q>n.\]
\end{lem}

The following lemma gives a statement on ODE comparison.
\begin{lem}\label{l2.2}
Let $a,b>0$. Assume that for some $\hat T\in(0,\yy]$ and $\hat\tau=\min\{1,\hat T/2\}$, the nonnegative functions $y\in C([0,\hat T))\cap C^1((0,\hat T))$, $f\in L_{loc}^1([0,\hat T))$ and satisfy
 \bes
 &y'(t)+ay(t)\le f(t),\ \ \ t\in(0,\hat T),&\label{2.2}\\[1mm]
 &\dd\int_t^{t+\hat\tau} f(s){\rm d}s\le b,\ \ \ t\in(0,\hat T-\hat\tau).&\nm\ees
Then
\bes
y(t)\le y(0)+2b+{b}/{a},\ \ \ t\in(0,\hat T).\label{2.4}
\ees
\end{lem}
\begin{proof}
If $\hat T>2$, then $\hat\tau=1$. Thanks to \cite[Lemma 3.4]{Stinner-S-Winkler}, we get \eqref{2.4}.

If $\hat T\le2$, then $\hat T=2\hat\tau$. For any $t_0\in[0,\hat T)$ and $t\in[t_0,t_0+\hat\tau]\cap[0,\hat T)$, integrating \eqref{2.2} from $t_0$ to $t$, there holds
\bes
y(t)\le y(t_0)+\int_{t_0}^tf(s)\ds\le y(t_0)+b.\label{2.5}
\ees
Hence, $y(t)\le y(0)+b$ for any $t\in[0,\hat\tau]$. Again by \eqref{2.5} and $\hat T=2\hat\tau$, we have
\[y(t)\le y(\hat\tau)+b\le y(0)+2b,\ \ t\in(\hat\tau,\hat T).\]
Therefore, for any $t\in(0,\hat T)$, there holds that $y(t)\le y(0)+2b$. This completes the proof.
\end{proof}

\section{Proofs of Theorem \ref{t1.1} and Theorem \ref{t1.1a}}
\setcounter{equation}{0} {\setlength\arraycolsep{2pt}
First of all, we introduce some notations. Let $r_*=\|r(x,t)\|_{L^\yy(\oo\times(0,\yy))}$ and
\bess
 &A=2\|u_0\|_\yy,\ \ B=2\|v_0\|_\yy,\ \ Q=\max\{\|w_0\|_\yy,r_*/\mu\},&\\[1mm]
&p=\min\kk\{k\in\mathbb{N}:\ k>{n}/{2}\rr\},\ \ \ \ma=\kk\{\eta,\lm,\mu,n,p,\oo\rr\},&\\[1mm]
&G_0=(A+B+1)(\|w_0\|_{W^1_{2(p+1)}(\oo)}+r_*),\ \ H_0=(A+B+1)Q+\|w_0\|_{W^2_{p+1}(\oo)}+r_*.&
\eess
Moreover, we define
\bess
T:=\sup\kk\{\tilde t\in(0,T_m):\,\|u(\cdot,t)\|_\yy\le A,\,\|v(\cdot,t)\|_\yy\le B \ \ {\rm in}\ (0,\tilde t)\rr\}.
\eess
Clearly, by the continuity of the solution, we have $T\in(0,T_m]$ and
\bes \sup_{t\in(0,T)}\|u(\cdot,t)\|_\yy\le A,\ \ \sup_{t\in(0,T)}\|v(\cdot,t)\|_\yy\le B.\label{3.1}\ees
We set from now on that $\tau:=\min\{1,T/2\}$. The following lemma provides the $L^{2(p+1)}$-bound of $\nn w$ in $(0,T)$.
\begin{lem}\label{l3.1}
Let $n\ge2$, $\eta_1=\eta_2=0$. Then there exists $C=C(\ma)>0$ such that
\bes
\ii|\nn w(\cdot,t)|^{2(p+1)}\le CG_0^{2(p+1)},\ \ t\in(0,T).\label{3.1a}
\ees
\end{lem}
\begin{proof}
According to the standard $L^p$-$L^q$ estimates for $(e^{t\tr})_{t\ge0}$ (\cite[Lemma 1.3 (ii)-(iii)]{M.W.2010}), we can find $\lm_1,\,C_1,\,C_2>0$ depending on $\ma$ such that, for all $t\in(0,T)$,
\bes
\|\nn w(\cdot,t)\|_{2(p+1)}&\le& C_1\|w_0\|_{W^1_{2(p+1)}(\oo)}+C_1\int_0^t\|\nn e^{(t-s)\tr}f(\cdot,s)\|_{2(p+1)}\ds\nm\\[0.5mm]
&\le&C_1\|w_0\|_{W^1_{2(p+1)}(\oo)}+C_2\int_0^t\kk(1+(t-s)^{-\frac12}\rr)
e^{-\lm_1(t-s)}\|f(\cdot,s)\|_{2(p+1)}\ds,\qquad\label{3.2a}
\ees
where $f=-\lm(u+v)w-\mu w+r$. It is easy to see from \eqref{3.1} and \eqref{2.1} that, there is $C_3=C_3(\ma)>0$ such that
\bess
\|f(\cdot,s)\|_{2(p+1)}\le C_3[(A+B+1)Q+r_*],\ \ \forall \ s\in(0,T).
\eess
Inserting this into \eqref{3.2a} and using the Sobolev embedding theorem: $W^1_{2(p+1)}(\oo)\hookrightarrow L^\yy(\oo)$ with $p+1>n/2+1$, one can find $C_4,C_5>0$ depending on $\ma$ such that
\bess
\|\nn w(\cdot,t)\|_{2(p+1)}\le C_4\big[(A+B+1)Q+\|w_0\|_{W^1_{2(p+1)}(\oo)}+r_*\big]
\le C_5G_0,\ \ t\in(0,T).
\eess
The proof is end.
\end{proof}

We proceed to find a space-time $L^{p+1}$ bound for $\tr w$.
\begin{lem}\label{l3.2}
Suppose that $n\ge2$ and $\eta_1=\eta_2=0$. Then there exists $C=C(\ma)>0$ such that
\bes
\int_t^{t+\tau}\ii|\tr w|^{p+1}\dx\ds\le CH_0^{p+1},\ \ t\in(0,T-\tau). \label{3.2}
\ees
\end{lem}
\begin{proof}
This proof is based on \cite[Lemma 4.2]{Lou-Winkler} and \cite[Lemma 4.3]{Blcak-Arxiv2019}. In view of the maximal Sobolev regularity properties of the Neumann heat semigroup $(e^{t\tr})_{t\ge0}$ (\cite{Giga-Sohr}) and \eqref{3.1}, we have
\bes
\int_0^{2\tau}\|w(\cdot,s)\|_{W^2_{p+1}(\oo)}^{p+1}\ds&\le&C_1\|w_0\|_{W^2_{p+1}(\oo)}^{p+1}
+C_1\int_0^{2\tau}\|f(\cdot,s)\|_{p+1}^{p+1}\ds\le C_2H_0^{p+1},\label{3.3}
\ees
where $f=-\lm(u+v)w-\mu w+r$ and $C_i=C_i(\ma)>0$ with $i=1,2$.

{\bf Case I: $T>2$}. In this case, it is easy to see from the definition of $\tau$ that $\tau=1$. Thus, \eqref{3.2} is equivalent to
\bes
\int_t^{t+1}\ii|\tr w|^{p+1}\dx\ds\le CH_0^{p+1},\ \ t\in(0,T-1). \label{3.3a}
\ees
From \eqref{3.3} with $\tau=1$, we know that \eqref{3.3a} holds for $t\in(0,1]$.

For $t_0\in(1,T-1)$, setting $\sigma:=t_0-1$ and hence $\sigma\in(0,T-2)$. Let $\rho\in C^\yy(\R)$ be an increasing function satisfying
\[0\le\rho\le1\ \ {\rm in}\ \R,\ \ \rho\equiv0\ \ {\rm in}\ (-\yy,0],\ \ \rho\equiv1\ \ {\rm in}\ (1,\yy),\]
and define $\rho_\sigma(t)=\rho(t-\sigma)$. Clearly, $\|\rho'\|_{C^1(\R)}\le C_3$ for some $C_3>0$.

Let $\vp$ be the unique classical solution of
\bess
 \left\{\begin{array}{lll}
 \vp_t=\tr\vp,&x\in\oo,\ \ t>\sigma,\\[1mm]
\pl_\nu\vp=0,\ \ &x\in\pl\oo, \ \ t>\sigma,\\[1mm]
 \vp(x,\sigma)=w(x,\sigma), \ \ &x\in\oo.
 \end{array}\right.
 \eess
It is easy to deduce that $\|\vp(\cdot,t)\|_\yy\le Q$ for $t\in(\sigma,\yy)$, and $\phi:=\rho_\sigma(t)\vp$ solves
\bess
 \left\{\begin{array}{lll}
\phi_t=\tr\phi+\rho_\sigma'(t)\vp, \ \ &x\in\oo,\ \ t>\sigma,\\[1mm]
 \pl_\nu\phi=0,\ \ &x\in\pl\oo, \ \ t>\sigma,\\[1mm]
 \phi(x,\sigma)=0, &x\in\oo.
 \end{array}\right.
 \eess
Thanks to the maximal Sobolev regularity properties of the Neumann heat semigroup $(e^{t\tr})_{t\ge0}$ (\cite{Giga-Sohr}), there holds
\bes
\int_\sigma^{\sigma+2}\|\phi(\cdot,s)\|_{W^2_{p+1}(\oo)}^{p+1}\ds\le  C_4\int_\sigma^{\sigma+2}\|\rho_\sigma'(s)\vp(\cdot,s)\|_{p+1}^{p+1}\ds
\le 2C_4|\oo|\|\rho'\|_{C^1(\R)}^{p+1} Q^{p+1}\le C_5 Q^{p+1},\quad\label{3.3b}
\ees
where positive constants $C_4=C_4(n,p,\oo)$ and $C_5=2C_3C_4|\oo|$. Noticing that $\phi(\cdot,t)=\vp(\cdot,t)$ for $t>\sigma+1$. Hence, we have from \eqref{3.3b} that
\bes
\int_{\sigma+1}^{\sigma+2}\|\vp(\cdot,s)\|_{W^2_{p+1}(\oo)}^{p+1}\ds\le C_5Q^{p+1},\ \ \sigma\in(0,T-2).\label{3.4}
\ees

Let $z(x,t)=w(x,t)-\vp(x,t)$ for $x\in\oo$ and $t\in[\sigma,T)$, then $z$ satisfies
\bess
 \left\{\begin{array}{lll}
 z_t=\tr z+f(x,t),&x\in\oo,\ \ t\in(\sigma,T),\\[1mm]
 \pl_\nu z=0,\ \ &x\in\pl\oo, \ \ t\in(\sigma,T),\\[1mm]
 z(x,\sigma)=0, &x\in\oo.
 \end{array}\right.
 \eess
Again by the maximal Sobolev regularity properties of the Neumann heat semigroup $(e^{t\tr})_{t\ge0}$ (\cite{Giga-Sohr}), one can find $C_6,C_7>0$ depending on $\ma$ such that
 \bes
\int_{\sigma}^{\sigma+2}\!\!\|z(\cdot,s)\|_{W^2_{p+1}(\oo)}^{p+1}\ds
\le C_6\int_{\sigma}^{\sigma+2}\!\!\|f(\cdot,s)\|_{p+1}^{p+1}\ds
\le C_7H_0^{p+1}. \label{3.5}
\ees
Note that
\[\|w(\cdot,s)\|_{W^2_{p+1}(\oo)}\le\|z(\cdot,s)\|_{W^2_{p+1}(\oo)}
+\|\vp(\cdot,s)\|_{W^2_{p+1}(\oo)},\ \ \forall\ s\in(\sigma+1,\sigma+2).\]
This combined with \eqref{3.4} and \eqref{3.5} yields that
\bess
\int_{\sigma+1}^{\sigma+2}\|w(\cdot,s)\|_{W^2_{p+1}(\oo)}^{p+1}\ds\le C_8H_0^{p+1},\ \ \sigma\in(0,T-2),
\eess
and hence
\bess
\int_{t_0}^{t_0+1}\|w(\cdot,s)\|_{W^2_{p+1}(\oo)}^{p+1}\ds\le C_8H_0^{p+1},\ \ t_0\in(1,T-1)
\eess
for the positive constant $C_8=C_8(\ma)$. We thus obtain \eqref{3.3a} for $t\in(1,T-1)$ due to the arbitrariness of $t_0$ in $(1,T-1)$. Recalling \eqref{3.3} with $\tau=1$, we get \eqref{3.3a}.

{\bf Case II: $T\le2$}. In this case $\tau=T/2$, i.e., $T=2\tau$. Due to \eqref{3.3}, we obtain \eqref{3.2}.
\end{proof}

Thanks to Lemmas \ref{l3.1}, \ref{l3.2}, we establish an uniform-in-time $L^4$ regularity and a space-time $L^6$ estimate for $\nn u$.
\begin{lem}\label{l3.3}
Let $n\ge2$ and $\eta_i=0$ with $i=1,2$. Then there exists $C=C(\ma)>0$ such that
\bes
\ii|\nn u(\cdot,t)|^{2p}\le C\|u_0\|_{W^1_{2p}(\oo)}^{2p}\big[(\chi G_0)^{2(p+1)}+(\chi H_0)^{p+1}+1\big],\ \ t\in(0,T).\label{3.6}
\ees
Moreover, there is $C'=C'(\chi,A,B,\ma)>0$ such that
\bes
\int_t^{t+\tau}\ii|\nn u|^{2(p+1)}\dx\ds\le C',\ \ t\in(0,T-\tau).\label{3.6b}
\ees
\end{lem}
\begin{proof}
We first recall from \cite[Lemma 2.2]{Lankeit-Wang} that, for $t\in(0,T_m)$,
\bes
\ii|\nn u|^{2(p+1)}\le2(n+4p^2)\|u\|_\yy^2\ii|\nn u|^{2(p-1)}|D^2u|^2,\label{3.5a}
\ees
and hence by \eqref{3.1}, we have
\bes
\ii|\nn u|^{2(p+1)}\le 2(n+4p^2)A^2\ii|\nn u|^{2(p-1)}|D^2u|^2=:k\ii|\nn u|^{2(p-1)}|D^2u|^2,\ \ t\in(0,T).\qquad\label{3.6a}
\ees
It follows from the first equation in \eqref{1.1} that
\bes
&&\df1{2p}\dv\ii|\nn u|^{2p}+\ii|\nn u|^{2p}\nm\\[0.5mm]
&=&\ii|\nn u|^{2(p-1)}\nn u\cdot\nn u_t+\ii|\nn u|^{2p}\nm\\[0.5mm]
&=&\ii|\nn u|^{2(p-1)}\nn u\cdot\nn(\tr u-\chi\nn\cdot(u\nn w))+\ii|\nn u|^{2p}\nm\\[0.5mm]
&=&\ii|\nn u|^{2(p-1)}\nn u\cdot\nn\tr u+\chi\ii\nn\cdot(|\nn u|^{2(p-1)}\nn u)(\nn\cdot(u\nn w))+\ii|\nn u|^{2p}\nm\\[0.5mm]
&=:&I(t)+J(t)+\ii|\nn u|^{2p},\ \ t\in(0,T_m).\label{3.7}
\ees
In view of \cite[Lemma 2.6 (ii)]{WjpW-CAMWA}, there is $C_1=C_1(n,p,\oo)>0$ such that
\bess
\int_{\pl\oo}|\nn u|^{2(p-1)}\pl_\nu|\nn u|^2{\rm d}S\le(p-1)\ii|\nn u|^{2(p-2)}\big|\nn|\nn u|^2\big|^2+C_1\ii|\nn u|^{2p},\ \ t\in(0,T_m).
\eess
Hence, by direct computations, we have
\bes
I(t)&=&\df12\ii|\nn u|^{2(p-1)}\tr|\nn u|^2-\ii|\nn u|^{2(p-1)}|D^2u|^2\nm\\[0.5mm]
&=&-\df12\ii\nn|\nn u|^{2(p-1)}\cdot\nn|\nn u|^2+\df12\int_{\partial\oo}|\nn u|^{2(p-1)}\partial_\nu|\nn u|^2{\rm d}S-\ii|\nn u|^{2(p-1)}|D^2u|^2\nm\\[0.5mm]
&=&-\df{p-1}2\ii|\nn u|^{2(p-2)}\kk|\nn|\nn u|^2\rr|^2+\df12\int_{\partial\oo}|\nn u|^{2(p-1)}\partial_\nu|\nn u|^2{\rm d}S-\ii|\nn u|^{2(p-1)}|D^2u|^2\nm\\[0.5mm]
&\le&-\ii|\nn u|^{2(p-1)}|D^2u|^2+\df{C_1}{2}\ii|\nn u|^{2p},\ \ t\in(0,T_m).\label{3.8}
\ees
Next, we estimate $J(t)$. It is easy to deduce that
\bes
J(t)&=&\chi\ii\kk(\nn|\nn u|^{2(p-1)}\cdot\nn u+|\nn u|^{2(p-1)}\tr u\rr)\kk(\nn u\cdot\nn w+u\tr w\rr)\nm\\[0.5mm]
&=&\chi\ii\kk(\nn|\nn u|^{2(p-1)}\cdot\nn u\rr)(\nn u\cdot\nn w)+\chi\ii u\tr w\kk(\nn|\nn u|^{2(p-1)}\cdot\nn u\rr)\nm\\[0.5mm]
&&+\chi\ii|\nn u|^{2(p-1)}\tr u\kk(\nn u\cdot\nn w\rr)+\chi\ii u|\nn u|^{2(p-1)}\tr u\tr w\nm\\[0.5mm]
&=:&J_1(t)+J_2(t)+J_3(t)+J_4(t),\ \ t\in(0,T_m).\label{3.9}
\ees
Noticing that $\nn|\nn u|^2=2D^2u\cdot\nn u$. Using Young's inequality: $ab\le |a|^q+|b|^{q'}$ with $q>1$ and $q'=q/(q-1)$, \eqref{3.1} and \eqref{3.6a}, the terms $J_1(t), J_2(t)$ in the right hand side of \eqref{3.9} can be estimated as
\bes
J_1(t)&=&\chi(p-1)\ii\kk(|\nn u|^{2(p-2)}\nn|\nn u|^2\cdot\nn u\rr)(\nn u\cdot\nn w)\nm\\[0.5mm]
&\le&\chi(p-1)\ii|\nn u|^{2(p-1)}|\nn w|\kk|\nn|\nn u|^2\rr|\nm\\[0.5mm]
&=&2\chi(p-1)\ii|\nn u|^{2(p-1)}|\nn w|\kk|D^2u\cdot\nn u\rr|\nm\\[0.5mm]
&\le&2\chi(p-1)\ii|\nn u|^{2p-1}|\nn w||D^2u|\nm\\[0.5mm]
&\le&\df18\ii|\nn u|^{2(p-1)}|D^2u|^2+8\chi^2(p-1)^2\ii|\nn u|^{2p}|\nn w|^2\nm\\[0.5mm]
&\le&\df18\ii|\nn u|^{2(p-1)}|D^2u|^2+\df{1}{8k}\ii|\nn u|^{2(p+1)}+8^{2p+1}k^p[\chi(p-1)]^{2(p+1)}\ii|\nn w|^{2(p+1)}\nm\\[0.5mm]
&\le&\df14\ii|\nn u|^{2(p-1)}|D^2u|^2+8^{2p+1}k^p[\chi(p-1)]^{2(p+1)}\ii|\nn w|^{2(p+1)}\nm\\[0.5mm]
&=&\df14\ii|\nn u|^{2(p-1)}|D^2u|^2+C_2A^{2p}\chi^{2(p+1)}\ii|\nn w|^{2(p+1)},\ \ t\in(0,T),\label{3.10}
\ees
where $C_2=8^{2p+1}[2(n+4p^2)]^p(p-1)^{2(p+1)}$ due to the definition of $k$, and
\bes
J_2(t)&=&\chi(p-1)\ii u|\nn u|^{2(p-2)}\tr w\kk(\nn|\nn u|^2\cdot\nn u\rr)\nm\\[0.5mm]
&\le&\chi(p-1)\ii u|\tr w||\nn u|^{2p-3}\kk|\nn|\nn u|^2\rr|\nm\\[0.5mm]
&\le&2\chi(p-1)\ii u|\tr w||\nn u|^{2(p-1)}|D^2 u|\nm\\[0.5mm]
&\le&\df18\ii|\nn u|^{2(p-1)}|D^2u|^2+8[\chi A(p-1)]^2\ii|\nn u|^{2(p-1)}|\tr w|^2\nm\\[0.5mm]
&\le&\df18\ii|\nn u|^{2(p-1)}|D^2u|^2+\df{1}{8k}\ii|\nn u|^{2(p+1)}+8^pk^{(p-1)/2}[\chi A(p-1)]^{p+1}\ii|\tr w|^{p+1}\nm\\[0.5mm]
&\le&\df14\ii|\nn u|^{2(p-1)}|D^2u|^2+8^pk^{(p-1)/2}[\chi A(p-1)]^{p+1}\ii|\tr w|^{p+1}\nm\\[0.5mm]
&=&\df14\ii|\nn u|^{2(p-1)}|D^2u|^2+C_3A^{2p}\chi^{p+1}\ii|\tr w|^{p+1},\ \ t\in(0,T)\label{3.11}
\ees
with $C_3=8^p\kk[2(n+4p^2)\rr]^{(p-1)/2}(p-1)^{p+1}$. In view of \eqref{3.1}, \eqref{3.6a} and the known inequality: $|\tr u|\le\sqrt{n}|D^2u|$, we estimate the last two terms in the right hand side of \eqref{3.9} as follows:
\bes
J_3(t)&\le&\sqrt{n}\chi\ii|\nn u|^{2p-1}|\nn w||D^2u|\nm\\[0.5mm]
&\le&\df18\ii|\nn u|^{2(p-1)}|D^2u|^2+2n\chi^2\ii|\nn u|^{2p}|\nn w|^2\nm\\[0.5mm]
&\le&\df18\ii|\nn u|^{2(p-1)}|D^2u|^2+\df{1}{8k}\ii|\nn u|^{2(p+1)}+(8k)^p\kk(2n\chi^2\rr)^{p+1}\ii|\nn w|^{2(p+1)}\nm\\[0.5mm]
&\le&\df14\ii|\nn u|^{2(p-1)}|D^2u|^2+(8k)^p\kk(2n\chi^2\rr)^{p+1}\ii|\nn w|^{2(p+1)}\nm\\[0.5mm]
&=&\df14\ii|\nn u|^{2(p-1)}|D^2u|^2+C_4A^{2p}\chi^{2(p+1)}\ii|\nn w|^{2(p+1)},\ \ t\in(0,T)\label{3.12}
\ees
with $C_4=\kk[16(n+4p^2)\rr]^p(2n)^{p+1}$, and
\bes
J_4(t)&\le&\sqrt{n}A\chi\ii|\nn u|^{2(p-1)}|D^2u||\tr w|\nm\\[0.5mm]
&\le&\df18\ii|\nn u|^{2(p-1)}|D^2u|^2+2nA^2\chi^2\ii|\nn u|^{2(p-1)}|\tr w|^2\nm\\[0.5mm]
&\le&\df18\ii|\nn u|^{2(p-1)}|D^2u|^2+\df{1}{16k}\ii|\nn u|^{2(p+1)}+(16k)^{\frac{p-1}2}(2nA^2\chi^2)^{\frac{p+1}2}\ii|\tr w|^{p+1}\nm\\[0.5mm]
&\le&\df{3}{16}\ii|\nn u|^{2(p-1)}|D^2u|^2+(16k)^{(p-1)/2}(2nA^2\chi^2)^{\frac{p+1}2}\ii|\tr w|^{p+1}\nm\\[0.5mm]
&=&\df{3}{16}\ii|\nn u|^{2(p-1)}|D^2u|^2+C_5A^{2p}\chi^{p+1}\ii|\tr w|^{p+1},\ \ t\in(0,T),\label{3.13}
\ees
where $C_5=[32(n+4p^2)]^{(p-1)/2}(2n)^{\frac{p+1}2}$. Inserting \eqref{3.10}-\eqref{3.13} into \eqref{3.9} we find that
 \bes
J(t)\le\df{15}{16}\ii|\nn u|^2|D^2u|^2+C'A^{2p}\chi^{2(p+1)}\ii|\nn w|^{2(p+1)}+C''A^{2p}\chi^{p+1}\ii|\tr w|^{p+1}\nm, \ \ t\in(0,T)
 \ees
with $C'=C_2+C_4$, $C''=C_3+C_5$. Plugging this and \eqref{3.8} into \eqref{3.7} yields that, for $t\in(0,T)$,
\bes
&&\df1{2p}\dv\ii|\nn u|^{2p}+\ii|\nn u|^{2p}+\df1{16}\ii|\nn u|^{2(p-1)}|D^2u|^2\nm\\[0.5mm]
&\le&C'A^{2p}\chi^{2(p+1)}\ii|\nn w|^{2(p+1)}+C''A^{2p}\chi^{p+1}\ii|\tr w|^{p+1}+\kk(\frac{C_1}{2}+1\rr)\ii|\nn u|^{2p}.\qquad\label{3.14}
\ees
By using Young's inequality and \eqref{3.6a}, one has
\bess
\kk(\frac{C_1}{2}+1\rr)\ii|\nn u|^{2p}&\le& \df{1}{32k}\ii|\nn u|^{2(p+1)}+(32k)^p\kk(\frac{C_1}{2}+1\rr)^{p+1}|\oo|\\[0.5mm]
&\le&\df1{32}\ii|\nn u|^{2(p-1)}|D^2u|^2+(32k)^p\kk(\frac{C_1}{2}+1\rr)^{p+1}|\oo|\\[0.5mm]
&=&\df1{32}\ii|\nn u|^{2(p-1)}|D^2u|^2+C_6A^{2p},\ \ t\in(0,T)
\eess
with $C_6=[64(n+4p^2)]^p\kk(\frac{C_1}{2}+1\rr)^{p+1}|\oo|$. Inserting this into \eqref{3.14} gives that, for $t\in(0,T)$,
\bes
&&\df1{2p}\dv\ii|\nn u|^{2p}+\ii|\nn u|^{2p}+\df1{32}\ii|\nn u|^{2(p-1)}|D^2u|^2\nm\\[0.5mm]
&\le&C'A^{2p}\chi^{2(p+1)}\ii|\nn w|^{2(p+1)}+C''A^{2p}\chi^{p+1}\ii|\tr w|^{p+1}+C_6A^{2p}\nm\\[0.5mm]
&\le&C_7A^{2p}\kk(\chi^{2(p+1)}\ii|\nn w|^{2(p+1)}+\chi^{p+1}\ii|\tr w|^{p+1}+1\rr)=:g(t),\ \ t\in(0,T),\label{3.16}
\ees
where $C_7=\max\{C',\,C'',\,C_6\}$ depending only on $n,p,\oo$. Applying \eqref{3.1a} and \eqref{3.2}, there exists $C_8=C_8(\ma)>0$ such that
\bess
\int_t^{t+\tau}g(s)\ds&\le& C_8A^{2p}\big[(\chi G_0)^{2(p+1)}+(\chi H_0)^{p+1}+1\big],\ \ t\in(0,T-\tau).
\eess
In view of Lemma \ref{l2.2} with $a=1$, $b=C_8A^{2p}\big[(\chi G_0)^{2(p+1)}+(\chi H_0)^{p+1}+1\big]$, $y(t)=\frac1{2p}\ii|\nn u|^4$ and $f(t)=g(t)$,
we find $C_9,C_{10}>0$ depending on $\ma$ such that
\bess
\ii|\nn u(\cdot,t)|^{2p}&\le& C_9A^{2p}\big[(\chi G_0)^{2(p+1)}+(\chi H_0)^{p+1}+1\big]+C_9\|\nn u_0\|_{2p}^{2p}\\[0.5mm]
&\le&C_{10}\|u_0\|_{W^1_{2p}(\oo)}^{2p}\big[(\chi G_0)^{2(p+1)}+(\chi H_0)^{p+1}+1\big],\ \ t\in(0,T),
\eess
where we have used $A=2\|u_0\|_\yy$ and the embedding: $W^1_{2p}(\oo)\hookrightarrow L^\yy(\oo)$ with $p>n/2$ in the derivation of the last inequality. This establishes \eqref{3.6}.

Integrating \eqref{3.16} from $t$ to $t+\tau$ with $t\in(0,T-\tau)$, and using \eqref{3.1a}, \eqref{3.2} and \eqref{3.6}, we have
\bess
\int_t^{t+\tau}\ii|\nn u|^{2(p-1)}|D^2u|^2\dx\ds\le C^*,\ \ t\in(0,T-\tau)
\eess
for some $C^*>0$ depending on $\chi,A,B,\ma$. This in conjunction with \eqref{3.6a} implies
\[\int_t^{t+\tau}\ii|\nn u|^{2(p+1)}\dx\ds\le \df{C^*}{2(n+4p^2)A^2},\ \ t\in(0,T-\tau).\]
Hence, we finally get \eqref{3.6b} and the proof is completed.
\end{proof}

\begin{remark}\label{r3.1}
From the proof of Lemma \ref{l3.3}, with some minor modifications, we see that for $\eta_1,\eta_2\ge0$ and $m,l>1$, the inequality \eqref{3.16} holds as well provided $u$ is uniformly bounded in time. And, the inequality \eqref{3.16} also plays an important role in the proofs of Theorems \ref{t1.2}, \ref{t1.3}.
\end{remark}

Based on Lemma \ref{l3.1}, we obtain uniform $L^\yy$ boundedness for $u$.
\begin{lem}\label{l3.4}
Let $n\ge2$, $\eta_i=0$ with $i=1,2$. Then there exists $C=C(\ma)>0$ such that
\bess
\|u(\cdot,t)\|_\yy\le \|u_0\|_\yy+CA\chi G_0,\ \ t\in(0,T).
\eess
\end{lem}
\begin{proof}
By using the standard $L^p$-$L^q$ estimates for $(e^{t\tr})_{t\ge0}$ (\cite[Lemma 3.3]{Fujie-I-W}), there exist positive constants $\lm_1,\,C_1$ depending on $n,\oo$ such that
\bess
\|u(\cdot,t)\|_\yy&\le& \|u_0\|_\yy+\chi\int_0^t\|e^{(t-s)\tr}\nn\cdot(u\nn w)\|_\yy\ds\nm\\[0.5mm]
&\le&\|u_0\|_\yy+C_1\chi\int_0^t\kk(1+(t-s)^{-\frac12-\frac{n}{ 4(p+1)}}\rr)e^{-\lm_1(t-s)}\|u\nn w\|_{2(p+1)}\ds,\ \ t\in(0,T).
\eess
Thanks to \eqref{3.1} and Lemma \ref{l3.1}, there is $C_2=C_2(\ma)>0$ such that
\[\|u(\cdot,\sigma)\nn w(\cdot,\sigma)\|_{2(p+1)}\le C_2AG_0,\ \ \sigma\in(0,T).\]
This in conjunction with the fact that $0<\frac12+\frac{n}{4(p+1)}<1$ due to $p+1>\frac{n}2+1$, one can find $C_3>0$ depending on $\ma$ fulfilling
\bess
\|u(\cdot,t)\|_\yy&\le&\|u_0\|_\yy+C_1C_2G_0\chi \int_0^t\kk(1+(t-s)^{-\frac12-\frac{n}{ 4(p+1)}}\rr)e^{-\lm_1(t-s)}\ds
\le\|u_0\|_\yy+C_3A\chi G_0
\eess
for all $t\in(0,T)$. The proof is finished.
\end{proof}

With Lemma \ref{l3.3} at hand, we derive an upper bound for $v$ in $(0,T)$.
\begin{lem}\label{l3.5}
Let $n\ge2$, $\eta_1=\eta_2=0$. Then there exists $C=C(\ma)>0$ such that, for $t\in(0,T)$,
\bess
\|v(\cdot,t)\|_\yy\le\|v_0\|_\yy+C\xi B\|u_0\|_{W^1_{2p}(\oo)}\kk[(\chi G_0)^{2(p+1)}+(\chi H_0)^{p+1}+1\rr]^{1/2p}.
\eess
\end{lem}
\begin{proof}
Thanks to the standard $L^p$-$L^q$ estimates for $(e^{t\tr})_{t\ge0}$ (\cite[Lemma 3.3]{Fujie-I-W}), there exist $\lm_1,\,C_1>0$ depending on $n,p$ such that
\bess
\|v(\cdot,t)\|_\yy&\le& \|v_0\|_\yy+\xi\int_0^t\|e^{(t-s)\tr}\nn\cdot(v\nn u)\|_\yy\ds\nm\\[0.5mm]
&\le&\|v_0\|_\yy+C_1\xi\int_0^t\kk(1+(t-s)^{-\frac12-\frac{n}{4p}}\rr)e^{-\lm_1(t-s)}\|v\nn u\|_{2p}\ds,\ \ t\in(0,T).
\eess
In view of \eqref{3.1} and \eqref{3.6}, there exists $C_2=C_2(\ma)>0$ such that
\bess
\|v(\cdot,\sigma)\nn u(\cdot,\sigma)\|_{2p}&\le& \|v(\cdot,\sigma)\|_\yy\|\nn u(\cdot,\sigma)\|_{2p}\\[0.5mm]
&\le&C_2B\|u_0\|_{W^1_{2p}(\oo)}\kk[(\chi G_0)^{2(p+1)}+(\chi H_0)^{p+1}+1\rr]^{1/2p},\ \ \sigma\in(0,T).
\eess
This combined the fact that $0<\frac12+\frac{n}{4p}<1$ due to $p>\frac{n}2$, we find $C_3=C_3(\ma)>0$ such that
\bess
\|v(\cdot,t)\|_\yy\le\|v_0\|_\yy+C_3\xi B\|u_0\|_{W^1_{2p}(\oo)}\kk[(\chi G_0)^{2(p+1)}+(\chi H_0)^{p+1}+1\rr]^{1/2p},\ \ t\in(0,T).
\eess
The proof is end.
\end{proof}

We finally show the global existence and boundedness of the classical solution of \eqref{1.1} under the condition that the initial data or $\chi,\xi$ are small enough.
\begin{lem}\label{l3.6}\, Let $n\ge2$ and $\eta_1=\eta_2=0$. There exists $\kappa=\kappa(\ma)>0$ such that, if $\chi,\xi$ and $(u_0,v_0,w_0)$ satisfy
\bes
\chi\le\df{\kappa}{G_0},\ \ {\rm and}\ \ \xi\le\df{\kappa}{\|u_0\|_{W^1_{2p}(\oo)}\kk[(\chi G_0)^{2(p+1)}+(\chi H_0)^{p+1}+1\rr]^{1/2p}},\label{3.23a}
\ees
then \qq{1.1} admits a unique nonnegative global solution $(u,v,w)\in [C(\bar\Omega\times [0,\yy))\cap C^{2,1}(\bar\Omega\times (0,\yy))]^3$, and
\bess
\|u(\cdot,t)\|_\yy\le A,\ \ \|v(\cdot,t)\|_\yy\le B,\ \ \|w(\cdot,t)\|_\yy\le Q,\ \ t\in(0,\yy).
\eess
Moreover, there exists $C=C(\chi,\xi,A,B,\ma)>0$ such that
\bes
\|u(\cdot,t)\|_{W^1_{2p}(\oo)}+\|v(\cdot,t)\|_{W^1_{2p}(\oo)}+\|w(\cdot,t)\|_{W^1_{2p}(\oo)}\le C,\ \ t\in(0,\yy).\label{3.24}
\ees
\end{lem}
\begin{proof}
We shall break this proof into three steps and use $C_i$ to denote the general positive constants that may depend on $\chi,\xi,\ma$ and the initial data.

{\it Step 1: The uniform boundedness of the solution in $(0,T_m)$.} From Lemma \ref{l3.4} and Lemma \ref{l3.5}, for any $\chi,\xi>0$, there exists $K=K(\ma)>0$ such that
\bes
\|u(\cdot,t)\|_\yy\le \|u_0\|_\yy+K\chi AG_0,\ \ t\in(0,T),\label{3.17}
\ees
and
\bes
\|v(\cdot,t)\|_\yy\le\|v_0\|_\yy+K\xi B\|u_0\|_{W^1_{2p}(\oo)}\kk[(\chi G_0)^{2(p+1)}+(\chi H_0)^{p+1}+1\rr]^{1/2p},\ \ t\in(0,T).\label{3.18}
\ees
If $(u_0,v_0,w_0)$ and $\chi,\xi$ satisfy
\bess
\chi\le\df{1}{4KG_0}\ \ {\rm and}\ \ \xi\le\df{1}{4K\|u_0\|_{W^1_{2p}(\oo)}\kk[(\chi G_0)^{2(p+1)}+(\chi H_0)^{p+1}+1\rr]^{1/2p}},
\eess
then, according to \eqref{3.17} and \eqref{3.18}, we find
\bess
\|u(\cdot,t)\|_\yy\le {3\|u_0\|_\yy}/{2}<A,\ \ \|v(\cdot,t)\|_\yy\le {3\|v_0\|_\yy}/{2}<B,\ \ t\in(0,T).
\eess
Due to the continuity of $u$ and $v$, there holds $T=T_m$ and
\bes
\|u(\cdot,t)\|_\yy\le A,\ \ \|v(\cdot,t)\|_\yy\le B,\ \ t\in(0,T_m).\label{3.18a}
\ees

{\it Step 2: The uniform $L^{2p}$-boundedness of $\nn v$ in $(0,T_m)$.}
Thanks to Lemmas \ref{l3.1} and \ref{l3.3} with $T$ replaced by $T_m$, there exists $C_1>0$ such that
\bes
\|\nn w(\cdot,t)\|_{2(p+1)}\le C_1,\ \ t\in(0,T_m),\label{3.19}
\ees
and
\bes
\|\nn u(\cdot,t)\|_{2p}\le C_1,\ \ t\in(0,T_m);\ \ \int_t^{t+\delta}\ii|\nn u|^{2(p+1)}\dx\ds\le C_1,\ \ t\in(0,T_m-\delta),\label{3.20}
\ees
where $\delta=\min\{1,T_m/2\}$. Moreover, from Lemma \ref{l3.2} with $T$ replaced by $T_m$, we find $C_2>0$ such that
\bes
\int_t^{t+\delta}\ii|\tr w|^{p+1}\dx\ds\le C_2,\ \ t\in(0,T_m-\delta).\label{3.21}
\ees

It is easy to see that $u$ satisfies
\bess
 \left\{\begin{array}{lll}
 u_t=\tr u+F(x,t),&x\in\oo,\ \ t\in(0,T_m),\\[1mm]
 \pl_\nu u=0,\ \ &x\in\pl\oo, \ \ t\in(0,T_m),\\[1mm]
 u(x,0)=u_0, &x\in\oo,
 \end{array}\right.
\eess
where $F(x,t)=\nn u\cdot\nn w+u\tr w$. In view of Young's inequality and the first inequality in \eqref{3.18a}, \eqref{3.19}, the second inequality in \eqref{3.20} and \eqref{3.21}, there is $C_3>0$ fulfilling
\bess
\int_t^{t+\delta}\ii |F(x,s)|^{p+1}\dx\ds\le C_3,\ \ t\in(0,T_m-\delta).
\eess
Hence, similar to the derivation of Lemma \ref{l3.2}, one can show that, there exists $C_4>0$ such that
\bes
\int_t^{t+\delta}\ii|\tr u|^{p+1}\dx\ds\le C_4,\ \ t\in(0,T_m-\delta).\label{3.22}
\ees

In view of the $L^\yy$ boundedness of $v$ in \eqref{3.18a}, following the arguments in the proof of Lemma \ref{l3.3} (cf. \eqref{3.16}), for some positive constant $C_5$, there holds
\bess
\dv\ii|\nn v|^{2p}+\ii|\nn v|^{2p}\le C_5\kk(\ii|\nn u|^{2(p+1)}+\ii|\tr u|^{p+1}+1\rr),
\eess
and then by using the second inequality in \eqref{3.20} and \eqref{3.22} and Lemma \ref{l2.2}, we have
\bes
\ii|\nn v(\cdot,t)|^{2p}\le C_6,\ \ t\in(0,T_m) \label{3.23}
\ees
for some $C_6>0$.

{\it Step 3: The global existence and boundedness of the solution.}
From \eqref{2.1}, \eqref{3.18a}, \eqref{3.19}, the first inequality in \eqref{3.20} and \eqref{3.23}, there is $C_7>0$ such that
\bess
\|u(\cdot,t)\|_{W^1_{2p}(\oo)}+\|v(\cdot,t)\|_{W^1_{2p}(\oo)}+\|w(\cdot,t)\|_{W^1_{2p}(\oo)}\le C_7,\ \ t\in(0,T_m).
\eess
This enables us to deduce that $T_m=\yy$ due to $2p>n$ and Lemma \ref{l2.1}. Moreover, \eqref{3.24} follows directly. This completes the proof.
\end{proof}

\begin{proof}[Proof of Theorems \ref{t1.1} and \ref{t1.1a}] Recalling the definitions of $G_0,H_0$ and taking
\[\chi_0=\df{\kappa}{G_0},\ \ {\rm and}\ \ \xi_0=\df{\kappa}{\|u_0\|_{W^1_{2p}(\oo)}\kk[(\chi G_0)^{2(p+1)}+(\chi H_0)^{p+1}+1\rr]^{1/2p}}\]
where $\kp=\kp(\ma)>0$ was given in Lemma \ref{l3.6}. On the one hand, for the fixed $\chi,\xi$, the small values of $\|w_0\|_{W^1_{2(p+1)}(\oo)}+r_*$ and $\|u_0\|_{W^1_{2p}(\oo)}$ ensure \eqref{3.23a}, then Theorem \ref{t1.1} follows from Lemma \ref{l3.6} directly. One the other hand, for the given initial data $(u_0,v_0,w_0)$, if $\chi\le\chi_0$ and $\xi\le\xi_0$, then we can show Theorem \ref{t1.1a} from Lemma \ref{l3.6}.
\end{proof}

\section{Proof of Theorem \ref{t1.2}}
  \setcounter{equation}{0} {\setlength\arraycolsep{2pt}
Define
\bess
\tf:=\sup\kk\{T\in(0,T_m):\,\|v(\cdot,t)\|_\yy\le B\ {\rm for\ all}\ t\in(0,T)\rr\}.
\eess
Evidently, by the continuity of the solution, we have $\tf\in(0,T_m]$ and
 \bes
\|v(\cdot,t)\|_\yy\le B,\ \ t\in(0,\tf).\label{4.1}
 \ees
For the later use, we denote $\te=\min\{1,\tf/2\}$ and
\[\mb=\{\lm,\mu,r_*,\eta,\chi,\oo,\|u_0\|_{W^1_4(\oo)},\|v_0\|_\yy,\|w_0\|_{W^2_3(\oo)}\}.\]
We first assert the uniform $L^\yy$ boundednss of $u$ provided \eqref{4.1}.
\begin{lem}\label{l4.1}
Let $n=2,\eta_1>0,m\ge2$ and $\eta_2=0$. Then there is $C=C(\mb)>0$ such that, the solution of \eqref{1.1} satisfies
\bes
\int_t^{t+\te}\ii|\nn u|^2\dx\ds\le C,\ \ t\in(0,\tf-\te).\label{4.7}
\ees
\end{lem}
\begin{proof} In the proof, we use $C_i$ to denote the positive constants which may depend on $\mb$.
Integrating the first equation in \eqref{1.1} upon $\oo$ yields
\bes
\dv\ii u=\eta_1\ii u-\eta_1\ii u^m,\ \ t\in(0,T_m),\label{4.3}
\ees
which implies
\bes
\ii u\le \max\{\|u_0\|_1,|\oo|\}=:C_1,\ \ t\in(0,T_m).\label{4.3a}
\ees
Integrating \eqref{4.3} from $t$ to $t+\te$ with $t\in(0,T_m-\te)$, and using the $L^1$ boundedness of $u$, we find
\bes
\int_t^{t+\te}\ii u^m\dx\ds\le C_1+\df{C_1}{\eta_1},\ \ t\in(0,T_m-\te).\label{4.4a}
\ees
Due to $m\ge2$, by use of Young's inequality, there is $C_2>0$ such that
\bes
\int_t^{t+\te}\ii u^2\dx\ds\le C_2,\ \ t\in(0,T_m-\te).\label{4.4}
\ees

Let $f(x,t)=-\lm(u+v)w-\mu w+r$. It follows from \eqref{4.1} and \eqref{4.4} that, there is $C_3>0$ such that
\[\int_t^{t+\te}\ii f^2\dx\ds\le C_3,\ \ t\in(0,\tf-\te).\]
As in the proof of Lemma \ref{l3.2}, one can show that, there exists $C_4>0$ such that
\bes
\int_t^{t+\te}\ii |\tr w|^2\dx\ds\le C_4,\ \ t\in(0,\tf-\te).\label{4.5}
\ees

Testing the $u$-equation in \eqref{1.1} by $u$, there holds
\bes
\df12\ii u^2+\ii|\nn u|^2&\le&\chi\ii u\nn u\cdot\nn w+\eta_1\ii u^2\nm\\[0.5mm]
&=&\df{\chi}2\ii \nn u^2\cdot\nn w+\eta_1\ii u^2\nm\\[0.5mm]
&=&-\df{\chi}2\ii u^2\tr w+\eta_1\ii u^2\nm\\[0.5mm]
&\le&\df{\chi}2\|u\|_4^2\|\tr w\|_2+\eta_1\ii u^2,\ \ t\in(0,T_m).\nm
\ees
In view of the Gagliardo-Nirenberg inequality (\cite{T-W2012}) and \eqref{4.3a}, there is $C_5>0$ such that
\bess
 \|u\|_4^2\le C_5(\|\nn u\|_2\|u\|_2+ \|u\|_1^2)\le C_5(\|\nn u\|_2\|u\|_2+ C_1^2),\ \ t\in(0,T_m).
\eess
Hence, for some $C_6>0$,
\bes
\dv\ii u^2+2\ii|\nn u|^2&\le&\chi C_5(\|\nn u\|_2\|u\|_2+ C_1^2)\|\tr w\|_2+2\eta_1\ii u^2\nm\\[0.5mm]
&\le& \|\nn u\|^2_2+C_6\big(\|u\|^2_2\|\tr w\|_2^2+\|u\|^2_2+\|\tr w\|_2^2+1\big),\ \ t\in(0,T_m).
\ees
By rearrangement, we have
 \bes
 z'(t)+\ii |\nn u|^2\le C_6\kk(\ii u^2dx+1\rr)\kk(\ii|\tr w|^2dx+1\rr)
 :=C_6z(t)h(t),\ \ t\in(0,\hat T),\label{4.6}
 \ees
where
 \[z(t)=\ii |u(\cdot,t)|^2+1, \ \ h(t)=\ii|\tr w(\cdot,t)|^2+1.\]
For any $0\le \tilde t\le t<\tf$, it follows from \eqref{4.6} that
\bes
z(t)\le z(\tilde t)e^{C_6\int_{\tilde t}^t h(s)\ds}.\label{4.6a}
\ees
Clearly, for $t\in[0,\te]$, from \eqref{4.6a} and \eqref{4.5} we have
\bes
z(t)\le z(0)e^{C_6\int_0^t h(s)\ds}\le (\|u_0\|_2^2+1)e^{C_6(C_4+1)}\label{4.8}
\ees

Fix $t\in(\te,\tf)$. Due to \eqref{4.4}, there exists $t_0\in[t-\te,t]$ such that
$z(t_0)=\ii u^2(\cdot,t_0)+1\le C_2+1$. Thanks to \eqref{4.5} we have
\[ \int_{t_0}^t h(s)\ds=\int_{t_0}^t\kk(\ii|\tr w|^2\dx+1\rr)\ds
\le\int_{t-\te}^t\kk(\ii|\tr w|^2\dx+1\rr)\ds\le C_4+1.\]
Again by \eqref{4.6a},
\bess
z(t)\le z(t_0)e^{C_6\int_{t_0}^t h(s)ds}\le (C_2+1)e^{C_6(C_4+1)}, \ \ \te<t<\tf.
\eess
This combined with \eqref{4.8} implies that $z(t)=\ii |u(\cdot,t)|^2+1\le C_7$ in $(0,\tf)$ with $C_7>0$. Moreover, it follows from \eqref{4.6} that
\bess
\dv\ii u^2+\ii|\nn u|^2\le C_6C_7\kk(\ii|\tr w|^2+1\rr),\ \ t\in(0,\tf),
\eess
which by an integration upon $(t,t+\te)$ for $t\in(0,\tf-\te)$ implies that, there is $C_8>0$ such that
\bess
\int_t^{t+\te}\ii|\nn u|^2\dx\ds\le C_8,\ \ t\in(0,\tf-\te).
\eess
This implies \eqref{4.7} and the proof is end.
\end{proof}

\begin{lem}\label{l4.2}
Suppose that $n=2$, $\eta_1>0,m\ge2$ and $\eta_2=0$. Then one can find $C=C(\mb)>0$ such that
\bes
\ii|\nn w(\cdot,t)|^4\le C,\ \ t\in(0,\tf).\label{4.12}
\ees
\end{lem}

\begin{proof}
By direct calculations (cf. \eqref{3.7} with $p=2$), there holds
\bes
&&\df14\dv\ii|\nn w|^4+\mu\ii|\nn w|^4\nm\\[0.5mm]
&=&\ii|\nn w|^2\nn u\cdot\nn w_t+\mu\ii|\nn w|^4\nm\\[0.5mm]
&=&\ii|\nn w|^2\nn w\cdot\nn(\tr w-\lm(u+v)w-\mu w+r)+\mu\ii|\nn w|^4\nm\\[0.5mm]
&=&\ii|\nn w|^2\nn w\cdot\nn\tr w-\lm\ii|\nn w|^2\nn w\cdot\nn(uw+vw)+\ii|\nn w|^2\nn w\cdot\nn r\nm\\[0.5mm]
&\le&\ii|\nn w|^2\nn w\cdot\nn\tr w-\lm\ii w|\nn w|^2\nn u\cdot\nn w-\lm\ii w|\nn w|^2\nn v\cdot\nn w+\ii|\nn w|^2\nn w\cdot\nn r\nm\\[0.5mm]
&=:&I_1(t)+I_2(t)+I_3(t)+I_4(t),\ \ t\in(0,T_m).\label{4.9}
\ees
In view of \cite[Lemma 2.2]{Lankeit-Wang}, we have
\bes
\ii|\nn w|^6\le 36Q^2\ii|\nn w|^2|D^2w|^2=:\tk\ii|\nn w|^2|D^2w|^2,\ \ t\in(0,T_m).\label{4.9a}
\ees

Making use of \cite[Lemma 2.6 (ii)]{WjpW-CAMWA}, there is $C_1>0$ depending on $n,p,\oo$ such that
\bess
\int_{\pl\oo}|\nn w|^2\pl_\nu|\nn w|^2{\rm d}S\le\df12\ii\big|\nn|\nn w|^2\big|^2+C_1\ii|\nn w|^4,\ \ t\in(0,T_m).
\eess
Hence, similar to the derivation of \eqref{3.8} with $p=2$, we get
\bes
I_1(t)&=&\df12\ii|\nn w|^2\tr|\nn w|^2-\ii|\nn w|^2|D^2w|^2\nm\\[0.5mm]
&=&-\df12\ii\kk|\nn|\nn w|^2\rr|^2+\df12\int_{\partial\oo}|\nn w|^2\partial_\nu|\nn w|^2{\rm d}S-\ii|\nn w|^2|D^2w|^2\nm\\[0.5mm]
&\le&-\ii|\nn w|^2|D^2w|^2-\df14\ii\kk|\nn|\nn w|^2\rr|^2+\df{C_1}{2}\ii|\nn w|^4,\ \ t\in(0,T_m).\nm
\ees
Thus, applying Young's inequality and \eqref{4.9a}, there holds
\bes
I_1(t)&\le&-\ii|\nn w|^2|D^2w|^2-\df14\ii\kk|\nn|\nn w|^2\rr|^2+\df{1}{4\tk}\ii|\nn w|^6+2\tk^2C_1^3|\oo|\nm\\[0.5mm]
&\le&-\df34\ii|\nn w|^2|D^2w|^2-\df14\ii\kk|\nn|\nn w|^2\rr|^2+2\tk^2C_1^3|\oo|,\ \ t\in(0,T_m).\label{4.9b}
\ees
By Young's inequality and \eqref{4.9a}, the second term in \eqref{4.9} can be estimated as:
\bes
I_2(t)&\le&\lm Q\ii|\nn w|^3|\nn u|\le \df{1}{4\tk}\ii |\nn w|^6+\tk\lm^2Q^2\ii|\nn u|^2\nm\\[0.5mm]
&\le&\df14\ii|\nn w|^2|D^2w|^2+\tk\lm^2Q^2\ii|\nn u|^2,\ \ t\in(0,T_m),\label{4.10}
\ees
and, the third term in \eqref{4.9} can be estimated as:
\bes
I_3(t)&=&\lm\ii v\nn\cdot(w|\nn w|^2\nn w)\nm\\[0.5mm]
&=&\lm\ii v|\nn w|^4+\lm\ii vw\nn|\nn w|^2\cdot\nn w+\lm\ii vw|\nn w|^2\tr w\nm\\[0.5mm]
&\le&\lm B\ii|\nn w|^4+\lm BQ\ii\kk|\nn|\nn w|^2\rr||\nn w|+\lm BQ\ii|\nn w|^2|\tr w|\nm\\[0.5mm]
&\le&\df18\ii\kk|\nn|\nn w|^2\rr|^2+\df1{4\tk}\ii|\nn w|^6+\ii|\tr w|^2+C_2\nm\\[0.5mm]
&\le&\df18\ii\kk|\nn|\nn w|^2\rr|^2+\df14\ii|\nn w|^2|D^2w|^2+\ii|\tr w|^2+C_2,\ \ t\in(0,\tf),\label{4.11}
\ees
where $C_2=C_2(\mb)>0$ and we have used \eqref{4.1} in the derivation of \eqref{4.11}. For the last term $I_4$, again by the Young inequality and \eqref{4.9a}, we find
\bes
I_4(t)&=&-\ii r\nn w\cdot\nn|\nn w|^2-\ii r|\nn w|^2\tr w\nm\\[0.5mm]
&\le&\df18\ii\kk|\nn|\nn w|^2\rr|^2+2r_*^2\ii |\nn w|^2+\ii|\tr w|^2+\df{r_*^2}{4}\ii |\nn w|^4\nm\\[0.5mm]
&\le&\df18\ii\kk|\nn|\nn w|^2\rr|^2+\ii|\tr w|^2+\df1{4\tk}\ii|\nn w|^6+C_3\nm\\[0.5mm]
&\le&\df18\ii\kk|\nn|\nn w|^2\rr|^2+\ii|\tr w|^2+\df14\ii|\nn w|^2|D^2w|^2+C_3,\ \ t\in(0,T_m)\label{4.11a}
\ees
for some $C_3=C_3(\mb)>0$. Plugging \eqref{4.9b}-\eqref{4.11a} into \eqref{4.9}, one can find $C_4=C_4(\mb)>0$ such that
\bess
\df14\dv\ii|\nn w|^4+\mu\ii|\nn w|^4\le \tk\lm^2Q^2\ii|\nn u|^2+2\ii|\tr w|^2+C_4,\ \ t\in(0,\tf).
\eess
In view of \eqref{4.7},\eqref{4.5} and Lemma \ref{l2.2}, we get \eqref{4.12}. This completes the proof.
\end{proof}

Now Lemma \ref{l4.2} enables us to establish the uniform $L^\yy$ boundedness of $u$.

\begin{lem}\label{l4.3} Let $n=2$, $\eta_1>0,m\ge2$ and $\eta_2=0$. Then there exists $C=C(\mb)>0$ such that
\bes
\|u(\cdot,t)\|_\yy\le C,\ \ t\in(0,\tf).\label{4.13}
\ees
\end{lem}
\begin{proof}
Denote $H(T)=\sup_{t\in(0,T)}\|u(\cdot,t)\|_\yy<\yy$ for $T\in(0,\tf)$.
Making use of the standard $L^p$-$L^q$ estimates for $(e^{t\tr})_{t\ge0}$ (\cite[Lemma 3.3]{Fujie-I-W}), one can find $\lm_1,\,C_1>0$ depending on $\oo$ such that
\bes
\|u(\cdot,t)\|_\yy&\le& \|u_0\|_\yy+\chi\int_0^t\|e^{(t-s)\tr}\nn\cdot(u\nn w)\|_\yy\ds+\eta_1\int_0^t\|e^{(t-s)\tr}(u-u^m)\|_\yy\ds\nm\\[0.5mm]
&\le&\|u_0\|_\yy+\chi\int_0^t\|e^{(t-s)\tr}\nn\cdot(u\nn w)\|_\yy\ds+\eta_1\int_0^t\|e^{(t-s)\tr}(u-u^m)_+\|_\yy\ds\nm\\[0.5mm]
&\le&C_1+C_1\chi\int_0^t\kk(1+(t-s)^{-\frac56}\rr)e^{-\lm_1(t-s)}\|u\nn w\|_3\ds,\ \ t\in(0,T).\label{4.14}
\ees
By \eqref{4.12}, \eqref{4.3a} and the definition of $H(T)$, there exists $C_2=C_2(\mb)>0$ such that
\bess
\|u\nn w\|_3\le\|u\|_{12}\|\nn w\|_4&=&\kk(\ii u^{12}\rr)^{1/12}\|\nn w\|_4
\le C_2H(T)^{11/{12}},\ \ t\in(0,T).
\eess
Inserting this into \eqref{4.14} yields that, for some $C_3=C_3(\mb)>0$,
\bess
\|u(\cdot,t)\|_\yy&\le&C_1+C_1C_2\chi H(T)^{11/{12}}\int_0^t\kk(1+(t-s)^{-\frac56}\rr)e^{-\lm_1(t-s)}\ds,\nm\\[0.5mm]
&\le&C_3+C_3H(T)^{11/{12}},\ \ t\in(0,T),
\eess
which implies
\[H(T)\le C_3+C_3H(T)^{11/{12}},\ \ T\in(0,\tf).\]
Hence, thanks to the Young inequality, we have $H(T)\le C_3^{12}+12C_3$ for all $ T\in(0,\tf)$. This combined with the definition of $H(T)$ finishes the proof.
\end{proof}

In order to obtain the $L^\yy$ estimate of $v$ in $(0,\tf)$, we shall need the $L^4$ regularity of $\nn u$.
\begin{lem}\label{l4.4}
Assume that $n=2$, $\eta_1>0,m\ge2$ and $\eta_2=0$. Then there exists $C=C(\mb)>0$ such that
\bes
\|\nn u(\cdot,t)\|_4\le C,\ \ t\in(0,\tf);\ \ \int_t^{t+\te}\ii|\nn u|^6\dx\ds\le C,\ \ t\in(0,\tf-\te).\label{4.16}
\ees
\end{lem}
\begin{proof}
Thanks to \eqref{4.1} and \eqref{4.13}, as in the proof of Lemmas \ref{l3.1} and \ref{l3.2}, there exists $C_1=C_1(\mb)>0$ such that
\bes
\ii |\nn w(\cdot,t)|^6\le C_1,\ \ t\in(0,\tf);\ \int_t^{t+\te}\ii|\tr w|^3\dx\ds\le C_1,\ \ t\in(0,\tf-\te).\label{4.17}
\ees
With \eqref{4.13} at hand, following the arguments in the proof of Lemma \ref{l3.3} with $n=p=2$, we obtain (cf. \eqref{3.16})
\bes
\dv\ii|\nn u|^4+\ii|\nn u|^4+\ii|\nn u|^2|D^2u|^2\le C_2\kk(\ii|\nn w|^6+\ii|\tr w|^3+1\rr),\ \ t\in(0,\tf)\qquad\label{4.17a}
\ees
for some $C_2=C_2(\mb)>0$. Then, using \eqref{4.17} and Lemma \ref{l2.2}, we get the first estimate in \eqref{4.16}.

Integrating \eqref{4.17a} from $t$ to $t+\te$ for $t\in(0,\tf-\te)$ firstly, and then applying the first inequality of \eqref{4.16} and \eqref{4.17} secondly, there is $C_3=C_3(\mb)>0$,
\[\int_t^{t+\te}\ii |\nn u|^2|D^2u|^2\dx\ds\le C_3,\ \ t\in(0,\tf-\te).\]
This in conjunction with \eqref{3.5a} with $n=p=2$ and \eqref{4.13} implies the second inequality of \eqref{4.16}.
\end{proof}

We are now in the position to get the $L^\yy$ estimate of $v$ in $(0,\tf)$.
\begin{lem}\label{l4.5}
Let $n=2$, $\eta_1>0,m\ge2$ and $\eta_2=0$. Then there is $C=C(\mb)>0$ such that
\bes
\|v(\cdot,t)\|_\yy\le \|v_0\|_\yy+C\xi,\ \ t\in(0,\tf).\label{4.18}
\ees
\end{lem}
\begin{proof}
Again by the standard $L^p$-$L^q$ estimates for $(e^{t\tr})_{t\ge0}$ (\cite[Lemma 3.3]{Fujie-I-W}), there exist positive constants $\lm_1,\,C_1$ depending on $\oo$ such that
\bess
\|v(\cdot,t)\|_\yy&\le& \|v_0\|_\yy+\xi\int_0^t\|e^{(t-s)\tr}\nn\cdot(v\nn u)\|_\yy\ds\nm\\[0.5mm]
&\le&\|v_0\|_\yy+C_1\xi\int_0^t\kk(1+(t-s)^{-\frac34}\rr)e^{-\lm_1(t-s)}\|v\nn u\|_4\ds,\ \ t\in(0,T_m).
\eess
In view of \eqref{4.1} and the first inequality of \eqref{4.16}, there exists $C_2=C_2(\mb)>0$ such that
\bess
\|v(\cdot,\sigma)\nn u(\cdot,\sigma)\|_4\le \|v(\cdot,\sigma)\|_\yy\|\nn u(\cdot,\sigma)\|_4\le C_2,\ \ \sigma\in(0,\tf).
\eess
Hence, one can find $C_3=C_3(\mb)$ such that \eqref{4.18} holds.
\end{proof}

\begin{proof}[Proof of Theorem \ref{t1.2}]
From Lemma \ref{l4.5}, we can see that, there is $K_1=K_1(\mb)>0$ such that, for any $\xi>0$,
\bes
\|v(\cdot,t)\|_\yy\le \|v_0\|_\yy+K_1(\mb)\xi,\ \ t\in(0,\tf).\label{4.19}
\ees
Noting that $K_1(\mb)$ is independent of $\xi$. Hence, if
\[\xi\le \df{\|v_0\|_\yy}{2K_1(\mb)},\]
then we have by \eqref{4.19} that $\|v(\cdot,t)\|_\yy\le 3\|v_0\|_\yy/2<B$ for all $ t\in(0,\tf)$.
Therefore, it follows from the definition of $\tf$ that $\tf=T_m$, and
$\|v(\cdot,t)\|_\yy\le B$ for all $t\in(0,T_m)$.
Moreover, Lemmas \ref{l4.1}-\ref{l4.5} holds with $\tf$ replaced by $T_m$. As in the proof of Lemma \ref{l3.6}, one can show that $\|\nn v(\cdot,t)\|_4<\yy$ for any $t\in(0,T_m)$, and hence
\[\|u(\cdot,t)\|_{W^1_4(\oo)}+\|v(\cdot,t)\|_{W^1_4(\oo)}+\|w(\cdot,t)\|_{W^1_4(\oo)}<\yy,\ \ t\in(0,T_m).\]
Thus, we have $T_m=\yy$ due to Lemma \ref{l2.1}. The proof Theorem \ref{t1.2} is completed.
\end{proof}

\section{Proof of Theorem \ref{t1.3}}
  \setcounter{equation}{0} {\setlength\arraycolsep{2pt}

Let $\delta=\min\{1,T_m/2\}$. Similar to the derivation of \eqref{4.3a} and \eqref{4.4a}, we get space-time $L^m$ (res. $L^l$) regularity for $u$ (res. $v$).
\begin{lem}\label{l5.1}
Let $n=2$. Suppose that $\eta_1,\eta_2>0$ and $m,l$ satisfy \eqref{1.8}. Then, there is $C>0$ such that
\bes
\ii u\le C,\ \ \ii v\le C,\ \ t\in(0,T_m),\label{5.1}
\ees
and
\bes
\int_t^{t+\delta}\ii u^m\dx\ds\le C,\ \ \int_t^{t+\delta}\ii v^l\dx\ds\le C,\ \ \ t\in(0,T_m-\delta).\label{5.2}
\ees
\end{lem}

Next we derive the space-time $L^2$ estimate for $\nn u$.
\begin{lem}\label{l5.2}
Let $n=2$, $\eta_1,\eta_2>0$ and $m,l$ satisfy \eqref{1.8}. Then, one can find $C>0$ such that
\bes
\int_t^{t+\delta}\ii |\nn u|^2\dx\ds\le C,\ \ \ t\in(0,T_m-\delta).\label{5.3}
\ees
\end{lem}
\begin{proof}
Noting that \eqref{5.2} holds with $m,l\ge2$, there is $C_1>0$ such that
\[\int_t^{t+\delta}\ii |f|^{\tilde m}\dx\ds\le C_1,\ \ t\in(0,T_m-\delta).\]
where $f(x,t)=-\lm(u+v)w-\mu w+r$ and $\tilde m=\min\{m,l\}\ge2$. As in the proof of Lemma \ref{l3.2}, there exists $C_2>0$ such that
\bes
\int_t^{t+\delta}\ii |\tr w|^{\tilde m}\dx\ds\le C_2,\ \ t\in(0,T_m-\delta).\label{5.4}
\ees
Hence, making use of the Young inequality, we have
\bes
\int_t^{t+\delta}\ii |\tr w|^2\dx\ds\le C_3,\ \ t\in(0,T_m-\delta)\label{5.4a}
\ees
for some $C_3>0$. Due to $m\ge2$ and the estimation for $u$ in \eqref{5.2}, the Young inequality says that, for some $C_4>0$,
\bes
\int_t^{t+\delta}\ii u^2\dx\ds\le C_4,\ \ t\in(0,T_m-\delta).\label{5.5a}
\ees
With the $L^1$-estimation of $u$ in \eqref{5.1}, \eqref{5.4a} and \eqref{5.5a} at hand, following the line of the proof of Lemma \ref{l4.1}, we get \eqref{5.3}.
\end{proof}

The following lemma asserts the uniform-in-time $L^4$ boundedness of $\nn w$.
\begin{lem}\label{l5.3}
Let $n=2$, $\eta_1,\eta_2>0$ and $m,l$ satisfy \eqref{1.8}. Then, one can find $C>0$ such that
\bes
\ii |\nn w(\cdot,t)|^4\le C,\ \ \ t\in(0,T_m).\label{5.5}
\ees
\end{lem}
\begin{proof}
The arguments are same with the proof of Lemma \ref{l4.2} except the estimation of $I_3$ in \eqref{4.11}. By use of Young's inequality, \eqref{4.9a} and \eqref{1.8}, we find
\bes
I_3(t)&=&\lm\ii v\nn\cdot(w|\nn w|^2\nn w)\nm\\[0.5mm]
&=&\lm\ii v|\nn w|^4+\lm\ii vw\nn|\nn w|^2\cdot\nn w+\lm\ii vw|\nn w|^2\tr w\nm\\[0.5mm]
&\le&\lm\ii v|\nn w|^4+\lm Q\ii v\kk|\nn|\nn w|^2\rr||\nn w|+\lm Q\ii v|\nn w|^2|\tr w|\nm\\[0.5mm]
&\le&(12\tk)^2\lm^3 \ii v^3+\df{1}{6\tilde k}\ii|\nn w|^6+\df18\ii\kk|\nn|\nn w|^2\rr|^2+2\lm^2Q^2\ii v^2|\nn w|^2\nm\\[0.5mm]
&&+(12\tk)^{1/2}(\lm Q)^{3/2}\ii v^{3/2}|\tr w|^{3/2}\nm\\[0.5mm]
&\le&\kk[(12\tk)^2\lm^3+(12\tk)^{1/2}2\sqrt{2}\lm^3 Q^3\rr] \ii v^3+\df{1}{4\tilde k}\ii|\nn w|^6+\df18\ii\kk|\nn|\nn w|^2\rr|^2\nm\\[0.5mm]
&&+C_1\ii v^{\frac{3\tilde m}{2\tilde m-3}}+C_1\ii|\tr w|^{\tilde m}\nm\\[0.5mm]
&\le&\df18\ii\kk|\nn|\nn w|^2\rr|^2+\df14\ii|\nn w|^2|D^2w|^2+C_1\ii|\tr w|^{\tilde m}+C_2\ii v^l+C_2,\ \ t\in(0,T_m).\qquad\label{5.6}
\ees
Inserting \eqref{4.9b},\eqref{4.10}, \eqref{5.6} and \eqref{4.11a} into \eqref{4.9}, one can find $C_3>0$ such that
\bess
\df14\dv\ii|\nn w|^4+\mu\ii|\nn w|^4\le C_3\kk(\ii|\nn u|^2+\ii|\tr w|^{\tilde m}+\ii|\tr w|^2+\ii v^l+1\rr),\ \ t\in(0,T_m).
\eess
Thanks to the second inequality in \eqref{5.2}, \eqref{5.3}, \eqref{5.4}, \eqref{5.4a} and Lemma \ref{l2.2}, we get \eqref{5.5}. This completes the proof.
\end{proof}

Noting that, we only use the uniform-in-time $L^1$ regularity of $u$ and $L^4$ estimate of $\nn w$ in the proof of Lemma \ref{l4.3}. Hence, by the similar arguments, we obtain the uniform-in-time $L^\yy$ boundedness of $u$.
\begin{lem}\label{l5.4}
Assume that $n=2$, $\eta_1,\eta_2>0$ and $m,l$ satisfy \eqref{1.8}. Then, there is $C>0$ such that
\bes
\|u(\cdot,t)\|_\yy\le C,\ \ \ t\in(0,T_m).\label{5.7}
\ees
\end{lem}
\begin{proof}
Following the line in the proof of Lemma \ref{l4.3}, using the $L^1$-estimate of $u$ in \eqref{5.1} and \eqref{5.5}, one can easily prove this lemma.
\end{proof}

The coming lemma provides the $L^4$ estimate for $\nn u$.
\begin{lem}\label{l5.5}
Assume that $n=2$, $\eta_1,\eta_2>0$ and $m,l$ satisfy \eqref{1.8}. Then, there is $C>0$ such that
\bes
\ii |\nn u(\cdot,t)|^4\le C,\ \ \ t\in(0,T_m),\label{5.8}
\ees
and
\bes
\int_t^{t+\delta}\ii |\nn u(\cdot,t)|^6\le C,\ \ \ t\in(0,T_m-\delta).\label{5.8a}
\ees
\end{lem}
\begin{proof}
By \eqref{5.7} and the second estimate in \eqref{5.2} with $l\ge3$ as well as Young's inequality, there is $C_1>0$ such that
\[\int_t^{t+\delta}\ii |f(x,s)|^3\dx\ds\le C_1,\ \ t\in(0,T_m-\delta).\]
where $f=-\lm(u+v)w-\mu w+r$. As in the proof of Lemma \ref{l3.2}, we find $C_2>0$ such that
\bes
\int_t^{t+\delta}\ii |\tr w|^3\dx\ds\le C_2,\ \ t\in(0,T_m-\delta).\label{5.9}
\ees
In view of the Gagliardo-Nirenberg inequality, there holds
\bes
\ii|\nn w(\cdot,t)|^6=\|\nn w(\cdot,t)\|_6^6&\le& C_3(\|\tr w(\cdot,t)\|_3^3\|w(\cdot,t)\|_\yy^3+\|w(\cdot,t)\|_\yy^6)\nm\\[0.5mm]
&\le&C_3Q^3\ii|\tr w(\cdot,t)|^3+C_3Q^6,\ \ t\in(0,T_m)\nm
\ees
for some $C_3>0$. This combined with \eqref{5.9} yields that, there is $C_4>0$ such that
\bes
\int_t^{t+\delta}\ii |\nn w|^6\dx\ds\le C_4,\ \ t\in(0,T_m-\delta).\label{5.10}
\ees
In view of \eqref{5.7}, following the proof of Lemma \ref{l3.3} with $n=p=2$, we get (cf. \eqref{3.16})
\bes
&&\dv\ii|\nn u(\cdot,t)|^4+\ii|\nn u(\cdot,t)|^4+\ii |\nn u(\cdot,t)|^2|D^2u(\cdot,t)|^2\nm\\[0.5mm]
&\le&C_5\kk(\ii|\nn w(\cdot,t)|^6+\ii|\tr w(\cdot,t)|^3+1\rr):=G(t),\ \ t\in(0,T_m)\label{5.11}
\ees
with some $C_5>0$. Due to \eqref{5.9} and \eqref{5.10}, it is easy to see that, there exists $C_6>0$ such that
\bes
\int_t^{t+\delta}G(s)\ds\le C_6,\ \ t\in(0,T_m-\delta).\label{5.12}
\ees
Making use of Lemma \ref{l2.2} with \eqref{5.11} and \eqref{5.12}, we have
\[\ii|\nn u(\cdot,t)|^4\le C_7,\ \ t\in(0,T_m).\]
This shows \eqref{5.8}. Moreover, integrating \eqref{5.11} from $t$ to $t+\delta$ for $t\in(0,T_m-\delta)$ firstly and using \eqref{5.8}, \eqref{5.12} secondly, it follows that, for some $C_7>0$,
\[\int_t^{t+\delta}\ii |\nn u|^2|D^2u|^2\dx\ds\le C_7,\ \ t\in(0,T_m-\delta).\]
Finally, by \eqref{3.5a} with $n=p=2$ and \eqref{5.7}, we get \eqref{5.8a}.
\end{proof}

In light of \eqref{5.8}, we show the uniform-in-time boundedness of $v$.
\begin{lem}\label{l5.6}
Let $n=2$, $\eta_1,\eta_2>0$ and $m,l$ satisfy \eqref{1.8}. Then, there exists $C>0$ such that
\bes
\|v(\cdot,t)\|_\yy\le C,\ \ \ t\in(0,T_m).\label{5.13}
\ees
\end{lem}
\begin{proof}
Thanks to uniform-in-time $L^4$ estimate for $\nn u$ in \eqref{5.8} and the $L^1$ regularity for $v$ in \eqref{5.1}, parallel to the arguments in the proof of Lemma \ref{l4.3}, we easily get the desired result.
\end{proof}

\begin{proof}[Proof of Theorem \ref{t1.3}]
We first note that, $u$ satisfies
\bess
 \left\{\begin{array}{lll}
 u_t=\tr u+F(x,t),&x\in\oo,\ \ t\in(0,T_m),\\[1mm]
 \pl_\nu u=0,\ \ &x\in\pl\oo, \ \ t\in(0,T_m),\\[1mm]
 u(x,0)=u_0, &x\in\oo,
 \end{array}\right.
\eess
where $F(x,t)=\nn u\cdot\nn w+u\tr w+\eta_1(u-u^m)$. In view of Young's inequality and \eqref{5.7}, \eqref{5.8a}, \eqref{5.9} and \eqref{5.10}, there exists $C_1>0$ such that
$\int_t^{t+\delta}\ii |F(x,s)|^3\dx\ds\le C_1$ for all $t\in(0,T_m-\delta)$.
Similar to the proof of Lemma \ref{l3.2}, we get that, for some $C_2>0$,
\bes
\int_t^{t+\delta}\ii|\tr u|^3\dx\ds\le C_2,\ \ t\in(0,T_m-\delta).\label{5.14}
\ees

With \eqref{5.13} at hand, by using the similar arguments in the proof of Lemma \ref{l3.3} with $n=p=2$ (cf. \eqref{3.16}), one can obtain that, for $C_3>0$,
\bess
\dv\ii|\nn v|^4+\ii|\nn v|^4\le C_3\kk(\ii|\nn u|^6+\ii|\tr u|^3+1\rr).
\eess
This in conjunction with \eqref{5.8a} and \eqref{5.14} and Lemma \ref{l2.2}, we have
\bes
\ii|\nn v(\cdot,t)|^4\le C_4,\ \ t\in(0,T_m) \label{5.15}
\ees
for some $C_4>0$. From the second estimate in \eqref{2.1}, \eqref{5.5}, \eqref{5.7}, \eqref{5.8}, \eqref{5.13} and \eqref{5.15}, there exists $C_5>0$ such that
\bess
\|u(\cdot,t)\|_{W^1_4(\oo)}+\|v(\cdot,t)\|_{W^1_4(\oo)}+\|w(\cdot,t)\|_{W^1_4(\oo)}\le C_5,\ \ t\in(0,T_m).
\eess
It follows from Lemma \ref{l2.1} that $T_m=\yy$. This completes the proof.
\end{proof}

\end{document}